\documentclass[a4paper,12pt,reqno]{amsart}
\usepackage[normalem]{ulem}   
\usepackage{amsmath,amssymb,amsthm,amsfonts}
\usepackage[active]{srcltx}
\usepackage{datetime}
\usepackage{enumerate}        
\usepackage[colorlinks,linkcolor={black},citecolor={black},urlcolor={black}]{hyperref}
\usepackage{color}            

\newcommand{\dd}{\, \mathrm{d}}
\newcommand{\trinorm}[1]{{\left\vert\kern-0.25ex\left\vert\kern-0.25ex\left\vert #1 
    \right\vert\kern-0.25ex\right\vert\kern-0.25ex\right\vert}}

\DeclareMathOperator{\Hess}{Hess}
\newcommand{\ip}[1]{\langle {#1}\rangle}

\newcommand{\D}{\mathcal{D}}

\newcommand{\R}{\mathbb{R}}

\newcommand{\G}{\mathcal{G}}
\newcommand{\eps}{\varepsilon}
\newcommand{\F}{\mathcal{F}}
\newcommand{\cL}{\mathcal{L}}

\renewcommand{\P}{{\mathbb P}}

\newcommand{\calP}{{\mathcal{P}}}
\newcommand{\calF}{{\mathcal{F}}}
\newcommand{\rmd}{{\rm d}}

\newcommand{\Ptwo}{\mathcal{P}_2}
\renewcommand{\div}{\mathop{\mathrm{div}}}

\newcommand{\AC}{\mathop{\mathrm{AC}}}
\DeclareMathOperator*\argmin{arg\,min\!}

\newtheorem{theorem}{Theorem}[section]
\newtheorem{lemma}[theorem]{Lemma}
\newtheorem{claim}[theorem]{Claim}

\theoremstyle{definition}

\newenvironment{remark}%
  {\par\medbreak\refstepcounter{theorem}%
    \noindent\textbf{Remark~\thetheorem. }}%
  {\qed\par\medskip}

\definecolor{jan}{rgb}{0.0,0.3,0.8}

\begin{document}

\title[From large deviations to Wasserstein gradient flows]{From large deviations to Wasserstein gradient flows in multiple dimensions}

\author{Matthias Erbar, Jan Maas and Michiel Renger}

\address{
University of Bonn\\
Institute for Applied Mathematics\\
Endenicher Allee 60\\
53115 Bonn\\
Germany}
\email{erbar@iam.uni-bonn.de}

\address{
Institute of Science and Technology Austria (IST Austria)\\
Am Campus 1\\ 
3400 \newline Klos\-ter\-neu\-burg\\ 
Austria}
\email{jan.maas@ist.ac.at}

\address{
WIAS Berlin\\ 
Mohrenstra\ss{}e 39\\ 
10117\\
Berlin\\ 
Germany}
\email{d.r.michiel.renger@wias-berlin.de}

\begin{abstract}
  We study the large deviation rate functional for the empirical
  distribution of independent Brownian particles with drift. In one
  dimension, it has been shown by Adams, Dirr, Peletier and Zimmer
  \cite{Adams2011} that this functional is asymptotically equivalent
  (in the sense of $\Gamma$-convergence as the
    time-step goes to zero) to the Jordan--Kinderlehrer--Otto
  functional arising in the Wasserstein gradient flow structure of the
  Fokker--Planck equation. In higher dimensions, part of this
  statement (the lower bound) has been recently proved by Duong,
  Laschos and Renger, but the upper bound remained open, since the
  proof in \cite{Duong2013a} relies on regularity properties of
  optimal transport maps that are restricted to one dimension. In this
  note we present a new proof of the upper bound, thereby generalising
  the result of \cite{Adams2011} to arbitrary
  dimensions.
\end{abstract}

\maketitle

\section{Introduction}
\label{sec:introduction}

In the recent paper \cite{Adams2011}, Adams, Dirr, Peletier and Zimmer
unveiled a fundamental connection between two seemingly unrelated
aspects of diffusion equations. They connected the \emph{large deviation
rate functional} for the empirical measure of a system of
independently diffusing particles to the \emph{entropy gradient flow
  structure} of diffusion equations in the Wasserstein space of
probability measures. Let us informally describe these two concepts
and their connection here, before giving rigorous statements in
Section~\ref{sec:main-results}.

\subsection*{Large deviations for independently diffusing particles}
We consider $n$ indistinguishable particles evolving according to the
stochastic differential equations
\begin{align}\label{eq:SDE}
  {\rm d} X_i(t) = - \nabla \Psi( X_i(t) )\,{\rmd} t +
  \sqrt{2}\,{\rmd} W_i(t)\;,
\end{align}
where $(W_1(t), \ldots, W_n(t))_{t \geq 0}$ is a collection of
independent standard $\R^d$-valued Brownian motions. We assume that
$\Psi : \R^d \to \R$ is twice continuously differentiable and that its
Hessian is uniformly bounded from below. Let
$\rho^{(n)}_t:=n^{-1}\sum_{i=1}^n\delta_{X_i(t)}$ denote the empirical
measure of $(X_{i}(t))_{i=1}^n$.  If the initial values $X_i(0)$ are
chosen deterministically such that $\rho^{(n)}_0$ converges weakly to some
fixed measure $\rho_0\in\calP(\R^d)$, then, for each
$t\geq0$, it is a classical result that the empirical measure $\rho^{(n)}_t$ converges almost surely to the unique solution of the Fokker-Planck equation
\begin{equation}\label{eq:FP}
  \partial_t\rho_t=\Delta\rho_t + \div(\rho_t\nabla\Psi)
\end{equation}
with initial condition $\rho_0$, see, e.g., \cite{Dawson1987,Feng2006} for much stronger results. 
Under suitable growth conditions on $\Psi$, a Sanov-type theorem implies that the random measures $(\rho_t^{(n)})_n$
satisfy a large deviation principle of the form
\begin{align*}
  \P[\rho_t^n \approx \bar\rho] \sim \exp{\big({-}n I_t(\bar \rho|\rho_0)\big)}\;,
\end{align*}
where the rate functional is given by
\begin{align}
  I_t(\bar\rho| \rho_0) := \inf_{\gamma \in \Gamma(\rho_0,\bar\rho)} H(\gamma | \rho_{0,t})\;,
\label{eq:rate-functional relative entropy form}
\end{align}
see \cite[Proposition~3.2]{Leonard2007} and \cite[Theorem~A.1]{Peletier2013}.
Here, $\rho_{0,t} \in \calP(\R^d \times \R^d)$ denotes the joint law
of a solution $(X_0, X_t)$ to \eqref{eq:SDE} with random initial
condition $X_0 \sim \rho_0$ (independent of the Brownian motion),
$H(\cdot|\rho_{0,t})$ denotes the relative entropy with respect to $\rho_{0,t}$, and
$\Gamma(\rho_0,\bar\rho)$ is the set of probability measures $\gamma \in
\calP(\R^d \times \R^d)$ with marginals $\rho_0$ and $\bar\rho$.
For background on large deviation theory we refer the reader to \cite{Dembo1998,Feng2006}.

In this paper we are interested in the short-time behaviour of the
rate functional $I_{t}(\cdot |\rho_{0})$ and its relation to the
Wasserstein gradient structure of the Fokker-Planck equation.

\subsection*{The Wasserstein gradient structure of the
  Fokker-Planck equation}
A seminal result by Jordan-Kinderlehrer-Otto \cite{Jordan1998} asserts
that the Fokker-Planck equation \eqref{eq:FP} can be regarded as the
gradient flow equation of the relative entropy
\begin{align*}
 \calF(\rho) &:= \left\{ \begin{array}{ll}\displaystyle
   \int_{\R^d} \rho(x) \log \rho(x) \dd x + \int_{\R^d} \Psi(x) \rho(x)\dd x
    &\rho({\rm d}x) = \rho(x)\,{\rm d} x\;, \\
 + \infty & \text{otherwise}\;,
\end{array} \right.
\end{align*}
in the Wasserstein space of probability measures $(\calP_2(\R^{d}),
W_{2})$. This result can be rigorously interpreted in different ways,
e.g., using the theory of gradient flows in metric spaces, or using an
infinite-dimensional Riemannian structure on the space of probability
measures; see \cite{Ambrosio2008} for details. Here we present the
original interpretation from \cite{Jordan1998} in terms of the
convergence of a discrete ``minimizing movement'' scheme, which can be
seen as an analogue of the implicit Euler scheme for the gradient flow
equation. For $\rho_{0} \in \calP_2(\R^{d})$ and $t > 0$, define
$J_{t}(\cdot | \rho_{0}) : \calP_2(\R^{d})\to \R \cup \{+\infty\}$ by
\begin{align}
J_{t}(\bar\rho | \rho_{0}) := \calF(\bar\rho) - \calF(\rho_0) + \frac{1}{2t} W_2(\rho_{0}, \bar\rho)^{2}\;,\quad\text{and set}\quad
S_{t}\lbrack\rho_{0}\rbrack := \argmin_{\bar\rho \in \calP(\R^{d})} \,J_{t}(\bar\rho | \rho_{0}) \;.
\label{eq:JKO scheme}
\end{align}
Since this minimisation problem has a unique minimiser, $S_{t}\lbrack
\rho_{0}\rbrack$ is well defined. The JKO-functional $J_{t}$ can be used to
construct an iterative discretisation scheme: it was shown in
\cite{Jordan1998} that
\begin{align*}
 \rho_{t} := \lim_{n \to \infty} \big(S_{{t/n}}\big)^{n}\lbrack\rho_{0}\rbrack
\end{align*}
exists for each $t > 0$ and satisfies the Fokker-Planck equation
\eqref{eq:FP}.

\subsection*{Relating $I_{t}$ and $J_{t}$} 
The main result of \cite{Adams2011} unveils a relation between the
large deviation principle and the Wasserstein gradient flow
structure. Roughly speaking, it asserts that the functionals $I_t$ and
$\frac12J_t$ are asymptotically equivalent as $t\to 0$. More
precisely, it was shown that
\begin{align}\label{eq:Gamma}
 I_{t}(\cdot|\rho_{0}) - \frac{1}{4t} W_{2}(\cdot,\rho_{0})^{2} \to 
 \frac12 \calF(\cdot) - \frac12 \calF(\rho) \qquad \text { as $t \to 0$}\;,
\end{align}
in the sense of $\Gamma$-convergence. This result provides an appealing microscopic explanation
for the emergence of the Wasserstein gradient flow structure at the
macroscopic level.

The proof of this theorem in \cite{Adams2011} required two strong
technical assumptions. Firstly, the result was limited to one space
dimension.  Secondly, the proof required highly restrictive regularity
assumptions on the involved measures. 

In a subsequent paper \cite{Duong2013a}, Duong, Laschos and Renger
were able to remove the strong regularity assumptions. Their approach
is based on a different representation of the rate functional $I_{t}$
due to Dawson and G\"artner \cite{Dawson1987} (see also
\cite{Feng2006}), that we shall describe in Section
\ref{sec:main-results}. The proof of the lower bound in the
$\Gamma$-convergence result in \cite{Duong2013a} is valid in arbitrary
dimensions. However, the remaining part of the argument (the
construction of a recovery sequence) is restricted to one dimension,
since it relies on regularity estimates for optimal transport maps
which are known to be false in multiple dimensions.
 
In this note we shall provide a different argument for the
construction of a recovery sequence that works in arbitrary
dimensions. 
Combined with the result from \cite{Duong2013a}, this completes the proof of \eqref{eq:Gamma} in arbitrary dimensions. We refer to Theorem \ref{th:main result} below for a precise statement.

\subsection*{Structure of the paper}
In Section \ref{sec:main-results} we give a detailed statement of the
main convergence result. In Section \ref{sec:ingredients} we collect
well-known results about Wasserstein gradient flows that will be used
in the proof. Section \ref{sec:proof} contains the proof of the
convergence result. For completeness, we also include the proof of the
lower bound taken from \cite{Duong2013a}.  In the appendix
we provide a short proof of the equivalence of different formulations
of the Benamou--Brenier formula.

\section{Statement of the main result}\label{sec:main-results}

In this section we shall rigorously introduce the three objects
appearing in the main result of this paper: the Wasserstein metric
$W_2$, the relative entropy functional $\F$, and the large deviation
rate functional $I_\tau$.

\subsubsection*{The Wasserstein metric}
Let $\calP_2(\R^d):=\{\rho\in\calP(\R^d):\int\lvert
x\rvert^2\,\rho({\rm d}x)<\infty\}$ denote the set of probability measures
with finite second moment. The $L^2$-Wasserstein distance between
$\rho_0,\rho_1\in\calP_2(\R^d)$ is defined by
\begin{align*}
 &W_2(\rho_0,\rho_1):=\inf_{\pi \in \Gamma(\rho_{0}, \rho_{1})}\bigg( \int_{\R^d \times \R^{d}} | x-y|^2\,\pi(\rmd x,\rmd y)\bigg)^{1/2}\;,
\end{align*}
where the infimum is taken over all couplings $\pi$ of $\rho_0$ and
$\rho_1$, i.e., $\Gamma(\rho_{0},\rho_{1})$ denotes the collection of
all $\pi\in\calP(\R^d\times\R^d)$ with $\pi(\cdot
\times\R^d)=\rho_0(\cdot)$ and $\pi(\R^d\times\cdot)=\rho_1(\cdot)$.

\subsubsection*{The relative entropy}
Throughout this paper we assume that $\Psi : \R^d \to \R$ is twice
continuously differentiable and $\lambda$-convex for some $\lambda \in
\R$, i.e., $\Hess\Psi(x) \geq \lambda\, \mathrm{Id}$ for all $x \in \R^d$.
The relative entropy functional $\F : \calP_2(\R^d) \to \R \cup
\{+\infty\}$ is defined by
\begin{align*}
 \F(\rho) :=  
 \left\{ \begin{array}{ll}\displaystyle
      \int_{\R^d} f(x)\log f(x)\dd x + \int_{\R^d} \Psi(x) f(x)\dd x
 & \text{if } \rho({\rm d} x)=f(x)\dd x\;,\\
+\infty
 & \text{otherwise}\;.\end{array} \right.
\end{align*}
This functional is well-defined, since the assumption on the second
moment implies that the negative parts of $f\log f$ and $\Psi f$ are
integrable with respect to the Lebesgue measure. If $\rho$ is
absolutely continuous with respect to the Lebesgue measure, then $\F$
can be written as a relative entropy with respect to the equilibrium
measure $\nu({\rm d}x) = e^{-\Psi(x)} \dd x$. Namely, 
 \begin{align*}
  \F(\rho) = 
   \int_{\R^d} g(x) \log g(x) \dd \nu(x)\;,
\end{align*}
where $\rho({\rm d} x)=g(x)\nu({\rm d}x)$.

We also introduce the relative Fisher information $\G : \calP_2(\R^d)
\to [0,+\infty]$ defined by
\begin{align*}
 \G(\rho) = 
 \left\{ \begin{array}{ll}\displaystyle
\int_{\{g>0\}} \frac{|\nabla g(x)|^2}{g(x)}\dd \nu(x)
& \text{if } \rho({\rm d}x)=g(x) \nu({\rm d}x),~g\in W_{\rm loc}^{1,1}(\R^d) \;,\\
+\infty
 & \text{otherwise}\;.\end{array} \right. 
\end{align*}

\subsubsection*{The large deviation rate functional}

The definition of the rate functional $I_\tau$ involves a weighted
Sobolev norm of negative order 1. Let $\D=C^\infty_c(\R^d)$ be the
space of test functions and let $\D'$ be the dual space of
distributions. Given $\rho\in\calP(\R^d)$, we define the weighted
$H^{-1}(\rho)$-norm of $s\in\D'$ by the duality formula
\begin{align*}
 \| s \|_{-1,\rho}^{2} := \sup_{f\in \D} \displaystyle \frac{\langle s,f\rangle^{2}}{\displaystyle \int_{{\R^{d}}} |\nabla f |^2\dd\rho}\;,
\end{align*}
where the supremum runs over all smooth test functions $f \in \D$ for
which the denominator does not vanish. Using the identity $b^{2}/a^{2}
= \sup_{t \in \R} 2t b - t^{2}a^{2}$, one obtains the equivalent
formula
\begin{align*}
\lVert s\rVert^2_{-1,\rho} =\sup_{f\in \D} \bigg\{ 2\langle s,f\rangle - \int\!\lvert\nabla f\rvert^2\dd\rho \bigg\}\;.
\end{align*} 

For fixed $\rho_0 \in \calP_2(\R^d)$ and $\tau > 0$, the functional
$I_\tau(\cdot| \rho_0) : \calP_2(\R^d) \to [0,+\infty]$ is 
defined by
\begin{align}\label{eq:rate-functional}
 I_\tau(\bar\rho| \rho_0) :=  \inf_{(\rho_t)_{t}\in\AC^2(\rho_0,\bar\rho)} \frac1{4\tau}\int_0^1 \big\|\partial_t\rho_t-\tau\Delta\rho_t-\tau\div(\rho_t\nabla\Psi)\big\|^2_{-1,\rho_t}\dd t\;,
\end{align}
where $\AC^2(\rho_0,\rho_1)$ denotes the set of 2-absolutely continuous curves
$(\rho_t)_{t\in[0,1]}$ in $\big(\calP_2(\R^d),W_2\big)$ with boundary
conditions $\rho|_{t = 0} = \rho_{0}$ and $\rho|_{t = 1} =
\rho_{1}$. We refer to Section \ref{sec:ingredients} for the definition of 2-absolutely continuity.
Intuitively, $I_\tau(\bar\rho| \rho_0)$ is the value of an
optimal control problem, which requires to interpolate between
$\rho_0$ and $\bar\rho$ in such a way that deviations from the
Fokker-Planck equation
\begin{align*}
 \partial_t\rho_t = \tau\Delta\rho_t + \tau\div(\rho_t\nabla\Psi)
\end{align*}
are minimised.

\begin{remark}\label{rem:def-consistent}
  Under two different sets of growth conditions on the potential
  $\Psi$, coined `subquadratic' and `superquadratic', the term inside
  the infimum of \eqref{eq:rate-functional} is the large deviation
  rate functional for trajectories $[0,\tau]\to\calP(\R^d)$ of the
  empirical measure of independent particles, see \cite{Dawson1987}.
    Using the contraction principle, it was proved in \cite[Cor.~4.10]{Duong2013a} that the  large deviation rate functional for the empirical measure at the end time $\tau$ is obtained by taking the infimum over (1-)absolutely continuous curves in $(\calP_2(\R^d), W_2)$ with the right boundary conditions. In the subquadratic case, it follows from the proof of \cite[Prop.~4.6]{Duong2013a} that if $\rho_0\in\calP_2(\R^d)$ and
    $\F(\rho_0)<\infty$, any weakly continuous curve with
    $\int_0^\tau\lVert \partial_t\rho_t-\Delta\rho_t-\div(\rho_t\nabla\Psi)\rVert^2\dd
    t<\infty$, is also 2-Wasserstein absolutely continuous. 
    In the superquadratic case, the same result was proved in
    \cite[Lem.~2.1]{Feng2012}. Therefore, under both sets of
    conditions on $\Psi$, we can take the infimum over 2-absolutely
    continuous curves in $(\calP_2(\R^d), W_2)$, hence the large deviation rate functional
    \eqref{eq:rate-functional relative entropy form} coincides with
    \eqref{eq:rate-functional}. In the rest of this paper we will not be
  concerned with the exact conditions under which these expressions
  coincide, but rather take \eqref{eq:rate-functional} as the object
  of study. For more details, see \cite[Section~4]{Duong2013a}.
\end{remark}

Now we are ready to state the main theorem of this paper:

\begin{theorem}[Main result]\label{th:main result}
  Let $\Psi \in C^2(\R^d)$ be $\lambda$-convex for some $\lambda \in
  \R$.  Then, for every $\rho_0\in\calP_2(\R^d)$ such that
  $\G(\rho_0)<\infty$, we have
  \begin{equation}\label{eq:expansion Mosco convergence}
  I_\tau(\,\cdot\mid\rho_0) - \frac{W_2^2(\rho_0,\,\cdot\,)}{4\tau} \xrightarrow[\tau\to0]{\Gamma} \frac12\F(\,\cdot\,) - \frac12\F(\rho_0)
\end{equation}
in the sense of $\Gamma$-convergence. More precisely:
\begin{itemize}
\item[(i)] For any $\rho_1 \in \calP_2(\R^n)$ and any sequence $\{\rho_1^\tau\}_\tau\subseteq\calP_2(\R^d)$ converging to $\rho_1$ in the $2$-Wasserstein metric, we have
\begin{equation}\label{eq:lower bound}
  \liminf_{\tau\rightarrow 0} \left(I_{\tau}(\rho_1^\tau\mid\rho_0)-\frac{W_{2}^{2}(\rho_{0},\rho_1^\tau)}{4\tau} \right)~\geq~ \frac{1}{2}\mathcal{F}(\rho_1)-\frac{1}{2}\mathcal{F}(\rho_{0})\;.
  \end{equation}
In addition, if
    $\nu(\R^d)=\int_{\R^d}e^{-\Psi(x)}\dd x<\infty$, then the lower
    bound~\eqref{eq:lower bound} also holds for any weakly
    converging sequence  $\{\rho_1^\tau\}_\tau\subseteq\calP_2(\R^d)$.
\item[(ii)] For any $\rho_1\in\Ptwo(\R^d)$ there exists a sequence
  $\{\rho_1^\tau\}_\tau\subseteq\calP_2(\R^d)$ converging to $\rho_1$ in the
  $2$-Wasserstein metric such that
\begin{equation}
  \label{eq:recovery sequence2}
  \limsup_{\tau\rightarrow 0}\left( I_{\tau}(\rho_1^\tau\mid\rho_0)-\frac{W_{2}^{2}(\rho_{0},\rho_1^\tau)}{4\tau}\right) ~\leq~ \frac{1}{2}\mathcal{F}(\rho_1)-\frac{1}{2}\mathcal{F}(\rho_{0})\;.
\end{equation}
\end{itemize}
\end{theorem}

As discussed in the introduction, this theorem was first proved in
dimension $1$ in \cite{Adams2011} under more restrictive conditions on
the measures $\rho_0$ and $\rho_1$. Part (i) has been extended to arbitrary dimensions in \cite{Duong2013a}. The novel contribution of our paper is a proof of (ii) in arbitrary dimensions.

\begin{remark}\label{rem:Entropy-finite}
  The right-hand side in \eqref{eq:lower bound} and
  \eqref{eq:recovery sequence2} is well-defined in $\R \cup \{ + \infty\}$, since our
  assumptions on $\rho_0$ imply that $\calF(\rho_0) < \infty$. This is
  a consequence of the HWI-inequality by Otto and Villani \cite{OV00}
  (see also \cite[Corollary 20.13]{Vil09}), which asserts that
  $\F(\rho) \leq W_2(\rho, \nu)\sqrt{\G(\rho)} -
  \frac{\lambda}{2}W_2^2(\rho,\nu)$.
\end{remark}

\section{Ingredients of the proof}
\label{sec:ingredients}

\subsection*{The Benamou--Brenier formula}

It will be convenient to work with the dynamic characterisation of the
Wasserstein distance due to Benamou--Brenier \cite{Benamou2000}, which 
asserts that, for $\rho_{0}, \rho_{1} \in \calP_{2}(\R^{d})$,
\begin{align}\label{eq:BB}
  W^2_2(\rho_0,\rho_1)~=~ \inf\limits_{(\rho_t)_t \in \AC^2(\rho_0,\rho_1)}\left\{\int_0^1\lVert \partial_t\rho_t\rVert^2_{-1,\rho_t}\dd t\right\}\;,
\end{align}
For $p \geq 1$, recall that a curve $(\rho_t)_{t\in[0,1]}$ is
said to be \emph{$p$-absolutely continuous with respect to $W_2$}, if there exists a scalar function $m \in L^p(0,1)$ satisfying
$W_2(\rho_s, \rho_t) \leq \int_s^t m(r) \dd r$ for all $0 \leq s < t
\leq 1$. We use the notation $(\rho_t)_t \in \AC^p(\rho_0,\rho_1)$. If $p = 1$, we simply say that $(\rho_t)_{t\in[0,1]}$ is
absolutely continuous. In this case, the metric derivative
\begin{align*}
  |\dot\rho_t| := \lim\limits_{h\to0}\frac{W_2(\rho_{t+h},\rho_t)}{h}
\end{align*}
exists for a.e. $t \in (0,1)$, see, e.g., \cite[Theorem
1.1.2]{Ambrosio2008} for more details.
It can be shown that \eqref{eq:BB} implies the identity
\begin{align}\label{eq:speed}
  |\dot\rho_t| =  \lVert \partial_t\rho_t\rVert_{-1,\rho_t}\;.
\end{align}
We refer to Appendix \ref{sec:BB} for an equivalent formulation of the
Benamou--Brenier formula which is commonly used in the literature on
optimal transport and to \cite[Theorem 8.3.1]{Ambrosio2008} for a proof
of \eqref{eq:BB}, \eqref{eq:speed} in this formulation.

\subsection*{Relative entropy, Fisher information, and heat flow}

A seminal result by McCann \cite{McCann1997} asserts that the
$\lambda$-convexity of $\Psi$ implies \emph{displacement
  $\lambda$-convexity} of $\F$, see also \cite[Theorem
5.15]{Villani2003}. This means that for any constant speed
$W_2$-geodesic $(\rho_t)_{t\in[0,1]} \subseteq \calP_2(\R^d)$ and any
$t \in [0,1]$, we have
\begin{align}\label{eq:displacement-convexity-1}
\F(\rho_t) &\leq (1-t)\F(\rho_0) +t\F(\rho_1)
   -\frac\lambda2 t(1-t)W^2_2(\rho_0,\rho_1)\;.
\end{align}
In particular, $\F$ is finite along geodesics as soon as it is finite
at the endpoints. The fact that the relative Fisher-information does
\emph{not} enjoy this property is the source of several complications
in \cite{Duong2013a}. We recall further that $\F$ is
  lower semicontinuous with respect to $W_2$-convergence, see
  \cite[Remark 9.4.2 and Lemma 9.4.3]{Ambrosio2008}.

The semigroup associated to the Fokker--Planck equation~\eqref{eq:FP}
will be denoted by $(P_t)_{t\geq0}$. More precisely, for
$\rho\in\calP_2(\R^d)$ we set $P_t \rho :=\rho_t$, where $(\rho_t)_t$ is
the unique distributional solution to the Fokker--Planck equation
\eqref{eq:FP} with $\rho_0=\rho$. This solution can be obtained using,
e.g., the metric theory of gradient flows for (generalised)
$\lambda$-convex functionals, see \cite[Thm. 11.2.8]{Ambrosio2008}.
 
In the following result we collect some well-known results on the
behaviour of the semigroup $(P_t)_{t \geq 0}$.

\begin{lemma}\label{lem:heat-flow-bounds}
The following assertions hold:
\begin{enumerate}
\item The curve $t\mapsto P_t\rho$ is continuous on $[0,\infty)$ and
  locally absolutely continuous on $(0,\infty)$ with respect to $W_2$.
\item For all $\rho,\sigma\in\calP_2(\R^d)$ and all $t \geq 0$ we have
  the contraction estimate:
  \begin{align}
    \label{eq:contraction}
    W_2(P_t\rho,P_t \sigma)&\leq e^{-\lambda t}W_2(\rho,\sigma)\;.
  \end{align} 
  Moreover, for any curve $(\rho_s)_{s}$ that is absolutely continuous
  with respect to $W_2$ we have
  \begin{align}
    \label{eq:action-bounds}
    \|\partial_s(P_t\rho_s)\|_{-1,P_t\rho_s} &\leq e^{-\lambda t}\|\partial_s\rho_s\|_{-1,\rho_s}\;.
  \end{align}

 \item 
  For all $\rho \in \calP_2(\R^d)$ and $t>0$ we have 
  \begin{align}\label{eq:instant-reg}
  \F(P_t\rho)< \infty\;,\quad\G(P_t\rho)<\infty\;,
 \end{align}
  as well as the bounds
  \begin{align}
    \label{eq:heat-flow-Fisher-bounds}
    \F(P_t \rho) \leq \F(\rho)\;,\quad
        \G(P_t\rho)  \leq e^{-2\lambda t}\G(\rho) \;.
  \end{align}
  Finally, for any $W_2$-geodesic $(\rho_s)_{s\in[0,1]}$ with
  $\F(\rho_0),\F(\rho_1)<\infty$,
  we have as $t\searrow0$:
  \begin{align}\label{eq:uniform-ent}
    \F(P_t\lbrack\rho_s\rbrack)~\nearrow~\F(\rho_s) \quad\text{ uniformly for } s\in[0,1]\;.
  \end{align}
  \end{enumerate}
\end{lemma}

\begin{proof}
  For part (1) and the properties \eqref{eq:contraction},
  \eqref{eq:instant-reg} and \eqref{eq:heat-flow-Fisher-bounds}, see
  \cite[Theorems 11.2.1 and 11.2.8]{Ambrosio2008}. The estimate
  \eqref{eq:action-bounds} follows immediately from
  \eqref{eq:contraction} and \eqref{eq:speed}. It remains to prove the
  statement \eqref{eq:uniform-ent}, which is less standard. Note first
  that by \eqref{eq:displacement-convexity-1} we have that $s\mapsto
  \F(\rho_s)$ is continuous and bounded. Our aim is to show that for
  every $\eps > 0$ there exists $\delta > 0$ such that
  $\F(\rho_{s})-\F(P_{t}\rho_{s}) < \eps$ whenever $t < \delta$ and $s
  \in [0,1]$.  Assume the contrary, i.e., that there exist $\eps>0$
  and sequences $t_k\to0$ and $(s_k)\subset[0,1]$ such that for all
  $k$,
  \begin{align}\label{eq:contra1}
    \F(\rho_{s_k})-\F(P_{t_k}\rho_{s_k})~\geq~\eps\;.
  \end{align}
  By compactness we can assume that $s_k\to s_0$ as $k\to\infty$ for some
  $s_0\in[0,1]$. We claim that $P_{t_k}\rho_{s_k}\to \rho_{s_0}$ in $W_2$-distance as $k \to \infty$. Indeed, again by
  \eqref{eq:contraction} the triangle inequality yields
  \begin{align*}
    W_2(\rho_{s_0},P_{t_k}\rho_{s_k})~&\leq~W_2(\rho_{s_0},P_{t_k}\rho_{s_0}) + W_2(P_{t_k}\rho_{s_0},P_{t_k}\rho_{s_k})\\
    &\leq~W_2(\rho_{s_0},P_{t_k}\rho_{s_0}) + e^{-\lambda t_k}W_2(\rho_{s_0},\rho_{s_k})\;,
  \end{align*}
  and the claim follows from the continuity of $P_t$ at $t=0$ and the
  continuity of the curve $(\rho_s)$. Passing to the limit $k\to\infty$ in
  \eqref{eq:contra1}, using the continuity of $s\mapsto\F(\rho_s)$ and the
  lower semicontinuity of $\F$ with respect to $W_2$, we obtain the following
  contradiction:
  \begin{align*}
    0~=~\F(\rho_{s_0})-\F(\rho_{s_0})~\geq~\limsup\limits_{k\to\infty}
    \Big(\F(\rho_{s_k}) - \F(P_{t_k}\rho_{s_k})\Big)~\geq~\eps\;,
  \end{align*}
  which completes the proof.
\end{proof}

We conclude this section by stating some useful identities for the derivative of the entropy.
In fact, for any absolutely continuous curve $(\rho_t)_{t\in[0,1]}$ with
$\F(\rho_t)\in\R$ for all $t$ and $\int_0^1\G(\rho_t)\dd t<\infty$ we
have that $t\mapsto\F(\rho_t)$ is absolutely continuous with
\begin{align}\label{eq:derivative F}
    \frac{\rm d}{{\rm d} t}\mathcal{F}(\rho_t)~=~ -\big\langle\partial_t\rho_t,\Delta\rho_{t}+\text{div}(\rho_{t}\nabla\Psi)
   \big\rangle_{-1,\rho_{t}}
\end{align}
for a.e. $t\in[0,1]$, see \cite[Lemma~2.3]{Duong2013a}. In particular,
if $\rho_t$ satisfies the Fokker-Planck equation we have
\begin{align}\label{eq:dissipation}
    -\, \frac{\rm d}{{\rm d} t}\mathcal{F}(\rho_t)
      = \| \Delta\rho_{t}+\text{div}(\rho_{t}\nabla\Psi)\|_{-1,\rho_t}^2= \G(\rho_t)\;,
\end{align}
where the second equality follows from~\eqref{eq:norm-explicit}.

\section{Proof of the main result}
\label{sec:proof}

\subsection{Upper bound}\label{sec:upper}

In this section we prove existence of the recovery sequence, i.e.,
statement (ii) of Theorem \ref{th:main result}. For this purpose we
define the set $Q:= \big\{\rho\in \mathcal{P}_{2}(\R^d):
\G(\rho)<\infty\big\}.$ Note that $\F(\rho) < \infty$ for all $\rho
\in Q$ in view of Remark \ref{rem:Entropy-finite}. Below we will prove
the following two claims:
\begin{claim}\label{claim:a}
For all $\rho_0, \rho_1\in Q$ we have as $\tau\to0$,
\begin{align} \label{eq:condition a}
  I_\tau(\rho_1\mid\rho_0)-\frac{1}{4\tau}W_2^2(\rho_{0},\rho_1) \to
  \frac12 \mathcal{F}(\rho_1)-\frac12\mathcal{F}(\rho_0)\;.
\end{align}
\end{claim}

\begin{claim}\label{claim:b}
  For every $\rho\in\calP_2(\R^d)$ there exists a sequence
  $(\rho^n)_{n} \subseteq Q$ such that $W_2^2(\rho^n,\rho)\to0$ and
  $\F(\rho^n)\to \F(\rho)$.
\end{claim}

The existence of the recovery sequence then follows from a
straightforward diagonal argument, see \cite[Proposition 6.2]{Duong2013a} for
details.

\begin{proof}[Proof of Claim \ref{claim:a}:] We only need to prove the
  limsup inequality for the left-hand side of \eqref{eq:condition a},
  since the liminf inequality will be proved in
  Section~\ref{sec:lower-bound} below.  If $\rho_0 = \rho_1$ the
  claim is immediate, so we take distinct measures $\rho_0, \rho_1 \in
  Q$, and take a geodesic $(\rho_{t})_{t\in[0,1]}$ connecting
  $\rho_{0}$ and $\rho_{1}$.  We will approximate this curve by
  running the semigroup for a small time $\eps = \eps(\tau)>0$, which
  will be determined below. A careful choice of $\eps$ as a function
  of $\tau$ is crucial for our argument.  We thus consider the curve
  $(\rho^\eps_t)_{t \in [0,1]}$ defined by
\begin{equation*}
  \rho^\eps_t~=~
  \begin{cases}
    P_{t}\rho_0\;,& 0\leq t\leq \eps\;,\\
    P_{\eps}\rho_{\frac{t-\eps}{1-2\eps}}\;, & \eps\leq t\leq 1-\eps\;,\\
    P_{1-t}\rho_1\;, & 1-\eps\leq t\leq 1\;.
  \end{cases}
\end{equation*}
For the sake of brevity, we shall write $\cL \rho = \Delta\rho +
\div(\rho\nabla\Psi)$.  Using the definition of
$I_\tau(\rho_1\mid\rho_0)$ and the second identity
\eqref{eq:dissipation}, we obtain 
\begin{align*}
&  I_{\tau}(\rho_{1}\mid\rho_0)-\frac{W_{2}^{2}(\rho_{0},\rho_{1})}{4\tau} 
\\&\leq \frac{1}{4\tau}\left(\int_0^1\!\left\lVert\partial_t\rho^\eps_t-\tau
   \cL\rho^\eps_{t}
   \right\rVert^2_{-1,\rho^\eps_{t}}\dd t-W_2^2(\rho_0,\rho_1)\right) \\
&  = 
   \frac{1}{4\tau}\left(\int_0^1\!\left\lVert\partial_t\rho^\eps_t\right\rVert^2_{-1,\rho^\eps_{t}}\dd t-W_2^2(\rho_0,\rho_1)\right)
- \frac12 \int_0^1\!\langle\partial_t\rho^\eps_t,\cL\rho^\eps_{t}    \rangle_{-1,\rho^\eps_{t}}\dd t
  +\frac{\tau}{4} \int_0^1\! \G(\rho^\eps_{t}) \dd t
     \;.
\end{align*}
We shall estimate these three terms separately. Let
$c_\lambda,k_\lambda > 0$ be sufficiently large so that
$\frac{e^{-2\lambda\eps}}{1-2\eps} \leq 1 + k_\lambda \eps$ and
$\int_0^\eps e^{-2\lambda t}\dd t\leq c_\lambda\eps$ for all $\eps \in
(0,\frac14)$. Using the semigroup estimates \eqref{eq:heat-flow-Fisher-bounds}
and \eqref{eq:action-bounds} and the Benamou--Brenier formula
\eqref{eq:BB}, the first term can be bounded by
\begin{align*}
    \int_0^1\!\left\lVert\partial_t\rho^\eps_t\right\rVert^2_{-1,\rho^\eps_{t}}\dd t
    &=\int_0^\eps\!\lVert\cL \rho^\eps_t\rVert^2_{-1,\rho^\eps_{t}} \dd t + \frac{1}{1-2\eps} \int_0^1\!\lVert \partial_t(P_\eps \rho_t)\rVert^2_{-1,P_\eps\rho_{t}}\dd t + \int_{1-\eps}^1
     \!\lVert\cL \rho^\eps_t\rVert^2_{-1,\rho^\eps_{t}} \dd t
    \\&\leq \int_0^\eps\! \G(P_t \rho_0) \dd t + \frac{e^{-2\lambda\eps}}{1-2\eps}\int_0^1\!\lVert \partial_t\rho_t\rVert^2_{-1,\rho_{t}}\dd t  + \int_{1-\eps}^1
     \! \G(P_{1-t}\rho_1) \dd t
  \\ &\leq c_\lambda\eps \G(\rho_0) + (1+k_\lambda \eps)W_2^2(\rho_0,\rho_1) +  c_\lambda\eps \G(\rho_1)\;.
\end{align*}
For the third term we use \eqref{eq:heat-flow-Fisher-bounds} to obtain
\begin{align*}
\int_0^1\! \G(\rho^\eps_{t}) \dd t \leq c_\lambda\eps (\G(\rho_0) + \G(\rho_1))
    +h(\eps)\;, \quad \text{where} \quad h(\eps) =  \int_0^1\!\G(P_\eps \rho_t) \dd t\;.
\end{align*}
We claim that $h(\eps)$ is finite for each $\eps > 0$. Indeed, using
\eqref{eq:heat-flow-Fisher-bounds} and \eqref{eq:dissipation} we
obtain
\begin{align}\label{eq:G-est}
  \G(P_\eps \rho_t) \int_{0}^\eps e^{2\lambda(\eps - s)} \dd s \leq
  \int_{0}^\eps \G(P_s \rho_t) \dd s = \F(\rho_t)-\F(P_\eps \rho_t)\;.
\end{align}
The right-hand side is uniformly bounded in $t$ thanks to the
$\lambda$-convexity of $\F$ and the uniform convergence
\eqref{eq:uniform-ent}. Consequently, $h(\eps) < \infty$.

To treat the second term, we can thus use \eqref{eq:derivative F} to obtain
\begin{align*}
\int_0^1\!\langle\partial_t\rho^\eps_t,\cL\rho^\eps_{t}
\rangle_{-1,\rho^\eps_{t}}\dd t = \F(\rho_0) - \F(\rho_1)\;.
\end{align*}
Combining these three bounds, we infer
that
\begin{align*}
  I_{\tau}(\rho_{1}\mid\rho_0)-\frac{W_{2}^{2}(\rho_{0},\rho_{1})}{4\tau}
   &  \leq
   \frac{1}{2}\F(\rho_{1})-\frac{1}{2}\mathcal{F}(\rho_{0}) 
   + \eps\frac{c_\lambda}{4} \Big(\tau+ \frac1\tau\Big) \big (\G(\rho_0) + \G(\rho_1) \big)
   \\& \qquad
      + \frac{k_\lambda \eps}{4\tau}    W_2^2(\rho_0,\rho_1) + \frac{\tau }4h(\eps) \;.
\end{align*}
We claim that $\eps=\eps(\tau)$ can be chosen as a function of $\tau$
such that
\begin{align}
  \frac{\eps(\tau)}{\tau}\to0 \quad\text{and} \quad \tau
  h\big(\eps(\tau)\big)\to 0 \quad \text {as } \tau \to 0\;.
  \label{eq:tau-eps}
\end{align}
This yields the limsup inequality in \eqref{eq:condition a}. The
corresponding liminf inequality will follow from
\eqref{eq:lower bound}.
   
It thus remains to prove the claim~\eqref{eq:tau-eps}. For
$\eps> 0$ we set
\begin{align*}
g(\eps):=\sqrt{\eps/h(\eps)}\;.
\end{align*}
Writing 
$g(\eps)=\sqrt{ {\eps e^{2\lambda\eps}}/{e^{2\lambda\eps}h(\eps)}}$,
it follows from \eqref{eq:heat-flow-Fisher-bounds} that $g$ is strictly
increasing on $(0,\eps_0)$ for $\eps_0$ sufficiently small. Taking into
account that $h(0) > 0$ since $\rho_0 \neq \rho_1$, we note that
$\lim_{\eps \to 0} g(\eps) = 0$. 
 To show that $g$ is right-continuous,
note that for each $t \in [0,1]$, the function $G_t : \eps \mapsto
\G(P_\eps \rho_t)$ is lower semicontinuous and non-negative, see
e.g. \cite[Proposition 10.4.14]{Ambrosio2008}. Fatou's lemma implies that $h :=
\int_0^1\!G_t \dd t$ is lower semicontinuous as well. Hence $g$ is upper
semicontinuous and thus right-continuous, since it is also
increasing. It follows from these properties that we can define
\begin{align*}
  \eps(\tau) := g^{-1}(\tau) := \inf \big\{\eps : g(\eps)>\tau\big\}
\end{align*}
as the generalised inverse of $g$. We shall show that this function
has the desired properties. 

Since $g$ is right-continuous, we note that $g(\eps(\tau)/2)
\leq \tau \leq g(\eps(\tau))$, which implies that the expressions in
\eqref{eq:tau-eps} can be estimated from above by
\begin{align*}
 \frac{\eps(\tau)}{\tau} \leq 2\sqrt{\frac{\eps(\tau)}{2} h\left(\frac{\eps(\tau)}{2}\right)}\;,\qquad 
 \tau h\big(\eps(\tau)\big)\leq \sqrt{\eps(\tau) h(\eps(\tau))}\;.
\end{align*}
It thus suffices to show that $\eps h(\eps) \to 0$ as $\eps \to 0$.
To show this, note that $\eps^{-1}\int_0^\eps e^{2\lambda s} \dd s
\geq \min \{1, e^{\lambda/2} \} =: \tilde k_\lambda$ for all $\eps \in
(0,\frac14)$.
Therefore, \eqref{eq:G-est} yields
\begin{align*}
  \tilde k_\lambda \eps \G(P_\eps \rho_t)
 \leq \F(\rho_t)-\F(P_\eps \rho_t)\;.
\end{align*}
By \eqref{eq:uniform-ent} the right-hand side converges to $0$ as
$\eps\to0$, uniformly for $t \in [0,1]$. It follows that
\begin{align*}
  \eps h(\eps) = \eps\int_0^1 \G(P_\eps \rho_t)\dd t~\to~ 0\
\end{align*}
as $\eps\to0$, which completes the proof.
\end{proof}

\begin{proof}[Proof of Claim \ref{claim:b}:]
  We approximate $\rho\in\calP_2(\R^d)$ by applying the
  semigroup. The first inequality in \eqref{eq:instant-reg} yield that $P_\eps\rho\in Q$ for
  any $\eps>0$, and Lemma \ref{lem:heat-flow-bounds}(1) implies that $P_\eps\rho$
  approximates $\rho$ in $W_2$-distance. Finally, since $\F$ is lower semicontinuous 
  with respect to $W_2$, the convergence $\F(P_\eps\rho)\to\F(\rho)$ as $\eps\to0$ follows from \eqref{eq:heat-flow-Fisher-bounds}.
\end{proof}

\subsection{Lower bound}
\label{sec:lower-bound}

For completeness, we reproduce here the short proof of statement (i) in Theorem 2.2, the lower bound, as given in [DLR13, Theorem 5.1, see erratum].

\begin{proof}
By definition of the infimum in \eqref{eq:rate-functional}, there
exists a sequence of absolutely continuous curves
$(\rho_{t}^\tau)_{t\in[0,1]}$ 
such that
\begin{align*}
  I_\tau(\rho_1^\tau\mid\rho_0)+\tau &\geq \frac{1}{4\tau}\int_{0}^{1}\!\big\lVert \partial_t \rho_t^\tau  -\tau(\Delta\rho_{t}^\tau+\div(\rho_{t}^\tau\,\nabla\Psi))\big\rVert^{2}_{-1,\rho_{t}^\tau}\dd t\;.
\end{align*}
In particular, the right-hand side is finite for all $\tau > 0$. Since $(\rho_t^\tau)_t$ is assumed to be $2$-absolutely continuous, we infer that  $\int_0^1\lVert\partial_t
\rho_t^\tau\rVert^{2}_{-1,\rho_{t}^\tau}\dd t$ is finite as well, and therefore
\begin{align*}
  \int_0^1 \G(\rho^\tau_t)\dd t = \int_0^1\lVert \Delta\rho_{t}^\tau+\div(\rho_{t}^\tau\,\nabla\Psi)\rVert^{2}_{-1,\rho_{t}^\tau}\dd t<\infty\;.
\end{align*}
It follows that that $t\mapsto \F(\rho^\tau_t)$
is absolutely continuous and the identity \eqref{eq:derivative F} holds.
Thus we can estimate
\begin{align*}
  I_\tau(\rho_1^\tau\mid\rho_0)+\tau &\geq \frac{1}{4\tau}\int_{0}^{1}\!\big\lVert \partial_t \rho_t^\tau  -\tau(\Delta\rho_{t}^\tau+\div(\rho_{t}^\tau\,\nabla\Psi))\big\rVert^{2}_{-1,\rho_{t}^\tau}\dd t\\
  &=\frac{1}{4\tau}\int_{0}^{1}\!\lVert \partial_t \rho_t^\tau\rVert^2_{-1,\rho_t^\tau} \dd t - \frac12\int_0^1\!\langle\partial\rho_t^\tau,\Delta\rho_t^\tau+\div(\rho_t^\tau\nabla\Psi)\rangle_{-1,\rho_t^\tau} \dd t\\
  &\hspace{5cm}+\frac\tau4\int_0^1\!\big\lVert \Delta\rho_t^\tau+\div(\rho_t^\tau\nabla\Psi)\big\rVert_{-1,\rho_t^\tau}^2 \dd t\\
  &\geq \frac{1}{4\tau}W_2^2(\rho_0,\rho_1^\tau) + \frac12\F(\rho_1^\tau) -
  \frac12\F(\rho_0)\;,
\end{align*}
where the last line follows from the Benamou--Brenier formula
\eqref{eq:BB}. The
claim~\eqref{eq:lower bound} then follows from the lower
semicontinuity of $\mathcal{F}$ with respect to $W_2$. 

If $\int_{\R^d}e^{-\Psi(x)\dd x}<\infty$, the result follows by applying the lower semicontinuity of $\F$ with respect to weak convergence in the final step.
\end{proof}

We finish by remarking that in the statement of Theorem \ref{th:main result}(i), the assumption of Wasserstein convergence cannot be weakened to weak convergence if the equilibrium measure $\nu$ does not have finite mass. A counterexample can be found in the erratum to \cite{Duong2013a}.

\appendix

\section{Equivalent formulations of the Benamou--Brenier formula}
\label{sec:BB}

The Benamou--Brenier formula in optimal transport asserts that for
$\rho_0,\rho_1\in\calP_2(\R^d)$,
\begin{align}\label{eq:BB-again}
  W^2_2(\rho_0,\rho_1)~=~ \inf\limits_{(\rho_t)_t\in\AC(\rho_0,\rho_1)}\left\{\int_0^1 \trinorm{ \partial_t\rho_t}^2_{-1,\rho_t}\dd t\right\}\;.
\end{align}
In this formula, the norm $\trinorm{\cdot}_{{-1,\rho}}$ is defined by
\begin{align}\label{eq:tri-norm}
 \trinorm{s}_{-1,\rho}^{2} := \inf_{v \in L^{2}(\rho;\R^{d})} 
    \bigg\{ \int_{\R^{d}} |v(x)|^{2} \dd \rho(x)  \ : \ s + \div (\rho v) = 0 \bigg\}\;.
\end{align}
for $\rho \in \calP(\R^{d})$ and $s \in \D'$.  It can be shown that
the infimum in this definition is uniquely attained, and its minimiser
can be characterised as follows: a solution $v \in L^2(\rho; \R^d)$ to
the ``continuity equation'' $s + \div (\rho v) = 0$ is optimal in
\eqref{eq:tri-norm} if and only if it belongs to the space of
generalised gradient vector fields defined by
\begin{align*}
 H_{\rho} := \overline{ \{ \nabla \psi : \R^{d} \to \R^{d}  \ |  \ \psi \in \D \}}^{L^{2}(\rho; \R^{d})}\;.
\end{align*}
We refer to \cite[Section 8.4]{Ambrosio2008} for the proof of these
facts. Note in particular that
\begin{align}\label{eq:norm-explicit}
  \trinorm{\div (\rho \nabla \psi)}_{-1,\rho}^2 = \int_{\R^d} |\nabla
  \psi(x)|^2 \dd \rho(x)
\end{align}
whenever $\nabla \psi \in L^2(\rho; \R^d)$.

The following lemma relates the norm $\trinorm{\cdot}_{{-1,\rho}}$ to
the norm $\lVert\cdot\rVert_{{-1,\rho}}$ defined in Section \ref{sec:main-results}.

\begin{lemma}\label{lem:compare-norms}
Let $\rho \in \calP(\R^{d})$ and $s \in \D'$. Then $\|s\|_{{-1,\rho}} =  \trinorm{s}_{-1,\rho}$.
\end{lemma}

\begin{proof}
  Suppose first that $\trinorm{s}_{-1, \rho} < \infty$, and let $v \in
  L^{2}(\rho;\R^{d})$ be the unique minimiser in the definition of
  $\trinorm{s}_{{-1,\rho}}$.  If $\trinorm{s}_{{-1,\rho}} = 0$, it
  follows that $v$ vanishes $\rho$-a.e., hence $\ip{s,f} = 0$ for all
  $f \in \D$, which implies that $\| s \|_{-1,\rho} = 0$.  Assume now,
  without loss of generality, that $\trinorm{s}_{{-1,\rho}}^{2} = \int
  |v|^{2} \dd \rho =1$. Then,
\begin{align*}
 \| s \|_{-1,\rho}
&    = \sup_{f \in \D} \bigg\{  \ip{-\div (\rho v), f}  \ \bigg| \ \int_{\R^{d}} |\nabla f|^{2} \dd \rho = 1 \bigg\}
\\&    = \sup_{f \in \D} \bigg\{  \int_{{\R^{d}}}  v\cdot\nabla f \dd \rho  \ \bigg| \ \int_{\R^{d}} |\nabla f|^{2} \dd \rho = 1 \bigg\}
\\&    = \sup_{f \in \D} \bigg\{  \frac12 \int_{{\R^{d}}}  | v |^{2} + | \nabla f |^{2} - | v - \nabla f|^{2}  \dd \rho  \ \bigg| \ \int_{\R^{d}} |\nabla f|^{2} \dd \rho = 1 \bigg\}
\\&       = \sup_{f \in \D} \bigg\{ 1 - \frac12 \int_{{\R^{d}}} | v - \nabla f|^{2}  \dd \rho  \ \bigg| \ \int_{\R^{d}} |\nabla  f|^{2} \dd \rho = 1 \bigg\}\;.
\end{align*}
Since $v\in H_{\rho}$, it follows from this computation that
$\|s\|_{-1,\rho} = 1 = \trinorm{s}_{-1,\rho}$.

On the other hand, if $\| s \|_{-1,\rho} < \infty$, it follows from
$\langle s,f\rangle\leq\lVert s\rVert_{-1,\rho}\cdot\lVert\nabla
f\rVert_{L^2(\rho;\R^d)}$ that the mapping
\begin{align*}
  T : \{ \nabla f : f \in \D\} \to \R, \qquad \nabla f \mapsto
  \ip{s,f}
\end{align*}
extends to a bounded linear functional
$T:(H_\rho,\lVert\cdot\rVert_{L^2(\rho;\R^d)})\to\R$ of norm $\| s
\|_{-1,\rho}$. Hence, the Riesz representation theorem implies that
$\ip{s,f} = \int_{\R^d}\!v\cdot\nabla f\dd \rho$ for some $v \in
H_\rho$ with $\|v\|_{L^2(\rho;\R^d)} = \| s \|_{-1,\rho}$. It follows
that $\trinorm{s}_{-1,\rho} \leq \|v\|_{L^2(\rho;\R^d)}$. In view of
the first part of the proof, the latter inequality is in fact an
equality.
\end{proof}

As a consequence of this lemma we infer that the Benamou--Brenier
formulas in \eqref{eq:BB} and \eqref{eq:BB-again} are equivalent.

\section*{Acknowledgement}
This work has been initiated when MR visited ME and JM at the
University of Bonn. The authors gratefully acknowledge support by the
German Research Foundation through the Hausdorff Center for
Mathematics at the University of Bonn, and the Collaborative Research
Centers 1060 ``The Mathematics of Emergent Effects'' and 1114
``Scaling Cascades in Complex Systems''. 

The authors thank the referees for their useful remarks and Kaveh Bashiri for pointing out a mistake in a preliminary version of this paper.

\bibliographystyle{alpha}
\bibliography{EMR-library}

\end{document}